\documentclass{llncs}
\usepackage{bm,amsmath,amssymb}
\newcommand{\lc}{\textrm{\upshape lc}}
\newcommand{\prs}{\textrm{\upshape prs}}
\newcommand{\rprs}{\textrm{\upshape rprs}}
\newcommand{\subres}{\textrm{\upshape S}}
\newcommand{\recsubres}{\bar{\textrm{\upshape S}}} 
\newcommand{\nessubres}{\tilde{\textrm{\upshape S}}}
\newcommand{\rednessubres}{\hat{\textrm{\upshape S}}}

\begin{document}

\title{Recursive Polynomial Remainder Sequence\\ and the Nested Subresultants}
\author{Akira Terui}
\institute{Graduate School of Pure and Applied Sciences\\
  University of Tsukuba\\
  Tsukuba, 305-8571, Japan\\
  \email{terui@math.tsukuba.ac.jp}
}
\maketitle

\begin{abstract}
  We give two new expressions of subresultants, \emph{nested
    subresultant} and \emph{reduced nested subresultant}, for the
  recursive polynomial remainder sequence (PRS) which has been
  introduced by the author.  The reduced nested subresultant reduces
  the size of the subresultant matrix drastically compared with the
  recursive subresultant proposed by the authors before, hence it is
  much more useful for investigation of the recursive PRS.  Finally,
  we discuss usage of the reduced nested subresultant in approximate
  algebraic computation, which motivates the present work.
\end{abstract}

\section{Introduction}
\label{sec:intro}

The polynomial remainder sequence (PRS) is one of fundamental tools in
computer algebra.  Although the Euclidean algorithm (see Knuth
\cite{knu1998} for example) for calculating PRS is simple,
coefficient growth in PRS makes the Euclidean algorithm often very
inefficient.  To overcome this problem, the mechanism of coefficient
growth has been extensively studied through the theory of
subresultants; see Collins \cite{col1967}, Brown and Traub
\cite{bro-tra71}, Loos \cite{loos1983}, etc.  By the theory of
subresultant, we can remove extraneous factors of the elements of PRS
systematically.

In our earlier research \cite{ter2003}, we have introduced a
variation of PRS called ``recursive PRS,'' and its subresultant called
``recursive subresultant.''  The recursive PRS is a result of repeated
calculation of PRS for the GCD and its derivative of the original PRS,
until the element becomes a constant.  Then, the coefficients of the
elements in the recursive PRS depend on the coefficients of the
initial polynomials.  By the recursive subresultants, we have given an
expression of the coefficients of the elements in the recursive PRS in
certain determinants of coefficients of the initial polynomials.
However, as the recursion depth of the recursive PRS has increased,
the recursive subresultant matrices have become so large that use of
them have often become impractical \cite{ter2004}.

In this paper, we give two other expressions of subresultants for the
recursive PRS, called ``nested subresultant'' and ``reduced nested
subresultant.''  The nested subresultant is a subresultant with
expression of ``nested'' determinants, used to show the relationship
between the recursive and the reduced nested subresultants.  The
reduced nested subresultant has the same form as the result of
Gaussian elimination with the Sylvester's identity on the nested
subresultant, hence it reduces the size of the subresultant matrix
drastically compared with the recursive subresultant.  Therefore, it
is much more useful than the recursive subresultant for investigation
of the recursive PRS.

This paper is organized as follows.  In Sect.~\ref{sec:recprs}, we
review the concept of the recursive PRS and the recursive
subresultant.  In Sect.~\ref{sec:nessubres}, we define the nested
subresultant and show its equivalence to the recursive subresultant.
In Sect.~\ref{sec:rednessubres}, we define the reduced nested
subresultant and show that it is a reduced expression of the nested
subresultant.  In Sect.~\ref{sec:disc}, we discuss briefly usage of
the reduced nested subresultant in approximate algebraic computation.

\section{Recursive PRS and Recursive Subresultants}
\label{sec:recprs}

Let $R$ be an integral domain and $K$ be its quotient field, and
polynomials $F$ and $G$ be in $R[x]$.  When we calculate PRS for $F$
and $G$ which have a nontrivial GCD, we usually stop the calculation
with the GCD.  However, it is sometimes useful to continue the
calculation by calculating the PRS for the GCD and its derivative;
this is used for square-free decompositions.  We call such a PRS a
``recursive PRS.''

To make this paper self-contained, we briefly review the definitions
and the properties of the recursive PRS and the recursive
subresultant, with necessary definitions of subresultants (for
detailed discussions, see Terui \cite{ter2003}).  In this paper, we
follow definitions and notations by von zur Gathen and L\"ucking
\cite{vzg-luc2003}.

\subsection{Recursive PRS}

\begin{definition}[Polynomial Remainder Sequence (PRS)]
  \label{def:prs}
  Let $F$ and $G$ be polynomials in $R[x]$ of degree $m$ and $n$
  ($m>n$), respectively.  A sequence $(P_1,\ldots,P_l)$ of nonzero
  polynomials is called a \emph{polynomial remainder sequence (PRS)}
  for $F$ and $G$, abbreviated to $\prs(F,G)$, if it satisfies
  $P_1=F$, $P_2=G$, $\alpha_i P_{i-2} = q_{i-1} P_{i-1} + \beta_i
  P_{i}$,
  for $i=3,\ldots,l$, where $\alpha_3,\ldots,\alpha_l,$
  $\beta_3,\ldots,\beta_l$ are elements of $R$ and
  $\deg(P_{i-1})>\deg(P_{i})$.  A sequence
  $((\alpha_3,\beta_3),\ldots,$ $(\alpha_l,\beta_l))$ is called a
  \emph{division rule} for $\prs(F,G)$.  If $P_l$ is a constant, then
  the PRS is called \emph{complete}.  \qed
\end{definition}

\begin{definition}[Recursive PRS]
  \label{def:recprs}
  Let $F$ and $G$ be the same as in Definition~\ref{def:prs}. Then, a
  sequence
    $(P_1^{(1)},\ldots,P_{l_1}^{(1)},
    P_1^{(2)},\ldots,P_{l_2}^{(2)},
    \ldots,
    P_1^{(t)},\ldots,P_{l_t}^{(t)})$
    of nonzero polynomials is called a \emph{recursive polynomial
      remainder sequence} (recursive PRS) for $F$ and $G$, abbreviated
   to $\rprs(F,G)$, if it satisfies
  \begin{equation}
    \label{eq:recprsdef}
    \begin{split}
      & P_1^{(1)} = F,\quad P_2^{(1)}=G,\quad
      P_{l_1}^{(1)}=\gamma_1\cdot\gcd(P_1^{(1)},P_2^{(1)})\quad
      \mbox{with $\gamma_1\in R$}, \\
      & (P_1^{(1)},P_2^{(1)},\ldots,P_{l_1}^{(1)})=\prs(P_1^{(1)},P_2^{(1)}),\\
      & P_1^{(k)}=P_{l_{k-1}}^{(k-1)},\;
      P_2^{(k)}=\frac{d}{dx}P_{l_{k-1}}^{(k-1)},\;
      P_{l_k}^{(k)}=\gamma_k\cdot\gcd(P_1^{(k)},P_2^{(k)})
      \;
      \mbox{with $\gamma_k\in R$}, \\
      & (P_1^{(k)},P_2^{(k)},\ldots,P_{l_k}^{(k)})=
      \prs(P_1^{(k)},P_2^{(k)}),
    \end{split}
  \end{equation}
  for $k=2,\ldots,t$.  If $\alpha_i^{(k)}$, $\beta_i^{(k)}\in R$
  satisfy
    $\alpha_i^{(k)} P_{i-2}^{(k)}
    =
    q_{i-1}^{(k)} P_{i-1}^{(k)}
    + 
    \beta_i^{(k)} P_i^{(k)}$
  for $k=1,\ldots,t$ and $i=3,\ldots,l_k$, then a sequence
  $((\alpha_3^{(1)},\beta_3^{(1)}),\ldots,
  (\alpha_{l_t}^{(t)},\beta_{l_t}^{(t)}))$ is called a
  \emph{division rule} for
  $\rprs(F,G)$. 
  Furthermore, if $P_{l_t}^{(t)}$ is a constant, then the recursive
  PRS is called complete.
  \qed
\end{definition}
  
In this paper, we use the following notations.  Let
$c_i^{(k)}=\lc(P_i^{(k)})$, $n_i^{(k)}=\deg(P_i^{(k)})$, $j_0=m$ and
$j_k=n_l^{(k)}$ for $k=1,\ldots,t$ and $i=1,\ldots,l_k$, and let
$d_i^{(k)}=n_i^{(k)}-n_{i+1}^{(k)}$ for $k=1,\ldots,t$ and
$i=1,\ldots,l_k-1$.

\subsection{Recursive Subresultants}

We construct ``recursive subresultant matrix'' whose determinants
represent the elements of the recursive PRS by the coefficients of
the initial polynomials.

Let $F$ and $G$ be polynomials in $R[x]$ such that
\begin{equation}
  \label{eq:fg}
  F(x) 
  = f_m x^m + \cdots + f_0 x^0,
  \quad
  G(x) 
  = g_n x^n + \cdots + g_0 x^0,
\end{equation}
with $m\ge n>0$.  For a square matrix $M$, we denote its determinant
by $|M|$.

\begin{definition}[Sylvester Matrix and Subresultant Matrix]
  Let $F$ and $G$ be as in \textup{(\ref{eq:fg})}.  The
  \emph{Sylvester matrix} of $F$ and $G$, denoted by $N(F,G)$ in
  (\ref{eq:sylmat}), is an $(m+n)\times(m+n)$ matrix constructed from
  the coefficients of $F$ and $G$.  For $j<n$, the \emph{$j$-th
    subresultant matrix} of $F$ and $G$, denoted by $N^{(j)}(F,G)$ in
  (\ref{eq:sylmat}), is an $(m+n-j)\times(m+n-2j)$ sub-matrix of
  $N(F,G)$ obtained by taking the left $n-j$ columns of coefficients
  of $F$ and the left $m-j$ columns of coefficients of $G$.
  \begin{equation}
    \small
    \label{eq:sylmat}
    \begin{split}
      N(F,G) &=
      \begin{pmatrix}
        f_m    &        &        & g_n    &        &  \\
        \vdots & \ddots &        & \vdots & \ddots &  \\
        f_0    &        & f_m    & g_0    &        & g_n \\
               & \ddots & \vdots &        & \ddots & \vdots \\
               &        & f_0    &        &        & g_0
      \end{pmatrix},
      \\
      \noalign{\vskip-10pt}
      &
      \hskip17pt
      \underbrace{\hspace{1.5cm}}_{n}
      \hskip2pt
      \underbrace{\hspace{1.3cm}}_{m}
    \end{split}
    \quad
    \begin{split}
      N^{(j)}(F,G) &=
      \begin{pmatrix}
        f_m    &        &        & g_n    &        &  \\
        \vdots & \ddots &        & \vdots & \ddots &  \\
        f_0    &        & f_m    & g_0    &        & g_n \\
               & \ddots & \vdots &        & \ddots & \vdots \\
               &        & f_0    &        &        & g_0
       \end{pmatrix}.
      \\
      \noalign{\vskip-10pt}
      &
      \hskip17pt
      \underbrace{\hspace{1.5cm}}_{n-j}
      \hskip2pt
      \underbrace{\hspace{1.3cm}}_{m-j}
    \end{split}
  \end{equation}
\end{definition}

\begin{definition}[Recursive Subresultant Matrix]
  \label{def:recsubresmat}
  Let $F$ and $G$ be defined as in \textup{(\ref{eq:fg})}, and let
  $(P_1^{(1)},\ldots,P_{l_1}^{(1)},\ldots,P_1^{(t)},\ldots,P_{l_t}^{(t)})$
  be complete recursive PRS for $F$ and $G$ as in
  Definition~\ref{def:recprs}.  Then, for each tuple of numbers
  $(k,j)$ with $k=1,\ldots,t$ and $j=j_{k-1}-2,\ldots,0$, define
  matrix $\bar{N}^{(k,j)}=\bar{N}^{(k,j)}(F,G)$ recursively as
  follows.
  \begin{enumerate}
  \item For $k=1$, let $\bar{N}^{(1,j)}(F,G)=N^{(j)}(F,G)$.
  \item For $k>1$, let $\bar{N}^{(k,j)}(F,G)$ consist of the upper
    block and the lower block, defined as follows:
    \begin{enumerate}
    \item The upper block is partitioned into $(j_{k-1}-j_k-1)\times
      (j_{k-1}-j_k-1)$ blocks with diagonal blocks filled with
      $\bar{N}_U^{(k-1,j_{k-1})}$, where $\bar{N}_U^{(k-1,j_{k-1})}$
      is a sub-matrix of $\bar{N}^{(k-1,j_{k-1})}(F,G)$ obtained by
      deleting the bottom $j_{k-1}+1$ rows.
    \item Let $\bar{N}_L^{(k-1,j_{k-1})}$ be a sub-matrix of
      $\bar{N}^{(k-1,j_{k-1})}$ obtained by taking the bottom
      $j_{k-1}+1$ rows, and let $\bar{N}_L^{'(k-1,j_{k-1})}$ be a
      sub-matrix of $\bar{N}_L^{(k-1,j_{k-1})}$ by multiplying the
      $(j_{k-1}+1-\tau)$-th rows by $\tau$ for
      $\tau=j_{k-1},\ldots,1$, then by deleting the bottom row.  Then,
      the lower block consists of $j_{k-1}-j-1$ blocks of
      $\bar{N}_L^{(k-1,j_{k-1})}$ such that the leftmost block is
      placed at the top row of the container block and the right-side
      block is placed down by 1 row from the left-side block, then
      followed by $j_{k-1}-j$ blocks of $\bar{N}_L^{'(k-1,j_{k-1})}$
      placed by the same manner as $\bar{N}_L^{(k-1,j_{k-1})}$.
    \end{enumerate}
  \end{enumerate}
  Readers can find the structures of $\bar{N}^{(k,j)}(F,G)$ in the
  figures in Terui~\cite{ter2003}.  Then, $\bar{N}^{(k,j)}(F,G)$ is
  called the \emph{$(k,j)$-th recursive subresultant matrix} of $F$
  and $G$.  \qed
\end{definition}

\begin{proposition}
  \label{prop:recsubresmat}
  For $k=1,\ldots,t$ and $j<j_{k-1}-1$, the numbers of rows and
  columns of $\bar{N}^{(k,j)}(F,G)$, the $(k,j)$-th recursive
  subresultant matrix of $F$ and $G$ are
      $(m+n-2j_1)
      \left\{
        \prod_{l=2}^{k-1}(2j_{l-1}-2j_l-1)
      \right\}
      (2j_{k-1}-2j-1)
      +j$,
      $(m+n-2j_1)
      \left\{
        \prod_{l=2}^{k-1}(2j_{l-1}-2j_l-1)
      \right\}
      (2j_{k-1}-2j-1)
      ,$
  respectively.
  \qed
\end{proposition}
\begin{definition}[Recursive Subresultant]
  \label{def:recsubres}
  Let $F$ and $G$ be defined as in \textup{(\ref{eq:fg})}, and let
  $(P_1^{(1)},\ldots,$
  $P_{l_1}^{(1)},\ldots,P_1^{(t)},\ldots,P_{l_t}^{(t)})$ be complete
  recursive PRS for $F$ and $G$ as in Definition~\ref{def:recprs}.
  For $j=j_{k-1}-2,\ldots,0$ and $\tau=j,\ldots,0$, let
  $\bar{N}_\tau^{(k,j)}=M_\tau^{(k,j)}(F,G)$ be a sub-matrix of the
  $(k,j)$-th recursive subresultant matrix \linebreak
  $\bar{N}^{(k,j)}(F,G)$ obtained by taking the top
  $(m+n-2j_1)\times\{\prod_{l=2}^{k-1}(2j_{l-1}-2j_l-1)\}(2j_{k-1}-2j-1)-1$
  rows and the
  $((m+n-2j_1)\{\prod_{l=2}^{k-1}(2j_{l-1}-2j_l-1)\}(2j_{k-1}-2j-1)
  +j-\tau)$-th row (note that $\bar{N}_\tau^{(k,j)}$ is a square
  matrix).  Then, the polynomial
    $\recsubres_{k,j}(F,G)
    =|\bar{N}_j^{(k,j)}|x^j+\cdots+|\bar{N}_0^{(k,j)}|x^0$
  is called the \emph{$(k,j)$-th recursive subresultant} of
  $F$ and $G$. \qed
\end{definition}

\section{Nested Subresultants}
\label{sec:nessubres}

Although the recursive subresultant can represent the coefficients of
the elements in recursive PRS, the size of the recursive subresultant
matrix becomes larger rapidly as the recursion depth of the recursive
PRS becomes deeper, hence making use of the recursive subresultant
matrix become more inefficient.

To overcome this problem, we introduce other representations for the
subresultant which is equivalent to the recursive subresultant up to a
constant, and more efficient to calculate.  The nested subresultant
matrix is a subresultant matrix whose elements are again determinants
of certain subresultant matrices (or even the nested subresultant
matrices), and the nested subresultant is a subresultant whose
coefficients are determinants of the nested subresultant matrices.

In this paper, the nested subresultant is mainly used to show the
relationship between the recursive subresultant and the reduced nested
subresultant, which is defined in the next section.  

\begin{definition}[Nested Subresultant Matrix]
  Let $F$ and $G$ be defined as in (\ref{eq:fg}), and let
  $(P_1^{(1)},\ldots,P_{l_1}^{(1)},\ldots,P_1^{(t)},\ldots,P_{l_t}^{(t)})$
  be complete recursive PRS for $F$ and $G$ as in
  Definition~\ref{def:recprs}.  Then, for each tuple of numbers
  $(k,j)$ with $k=1,\ldots,t$ and $j=j_{k-1}-2,\ldots,0$, define
  matrix $\tilde{N}^{(k,j)}(F,G)$ recursively as follows.
  \begin{enumerate}
  \item For $k=1$, let $\tilde{N}^{(1,j)}(F,G)=N^{(j)}(F,G)$.
  \item For $k>1$ and $\tau=0,\ldots,j_{k-1}$, let
    $\tilde{N}_\tau^{(k-1,j_{k-1})}$ be a sub-matrix of
    $\tilde{N}^{(k-1,j_{k-1})}$ by taking the top
    $(n_1^{(k-1)}+n_2^{(k-1)}-2j_{k-1}-1)$ rows and the
    $(n_1^{(k-1)}+n_2^{(k-1)}-j_{k-1}-\tau)$-th row (note that
    $\tilde{N}_\tau^{(k-1,j_{k-1})}$ is a square matrix).  Now, let
    \begin{equation}
      \tilde{N}^{(k,j)}(F,G)=
      N^{(j)}
      \left(
        \nessubres_{k-1,j_{k-1}}(F,G),
        \frac{d}{dx}\tilde{\subres}_{k-1,j_{k-1}}(F,G)
      \right),
    \end{equation}
    where $\tilde{\subres}_{k-1,j_{k-1}}(F,G)$ is defined by
    Definition~\ref{def:nessubres}.
    Then, $\tilde{N}^{(k,j)}(F,G)$ is called the \emph{$(k,j)$-th
      nested subresultant matrix of $F$ and $G$}.  \qed
  \end{enumerate}
\end{definition}

\begin{definition}[Nested Subresultant]
  \label{def:nessubres}
  Let $F$ and $G$ be defined as in \textup{(\ref{eq:fg})}, and let
  $(P_1^{(1)},\ldots,$
  $P_{l_1}^{(1)},\ldots,P_1^{(t)},\ldots,P_{l_t}^{(t)})$ be complete
  recursive PRS for $F$ and $G$ as in Definition~\ref{def:recprs}.
  For $j=j_{k-1}-2,\ldots,0$ and $\tau=j,\ldots,0$, let
  $\tilde{N}_\tau^{(k,j)}=\tilde{N}_\tau^{(k,j)}(F,G)$ be a sub-matrix
  of the $(k,j)$-th nested recursive subresultant matrix
  $\tilde{N}^{(k,j)}(F,G)$ obtained by taking the top
  $n_1^{(k)}+n_2^{(k)}-2j-1$ rows and the
  $(n_1^{(k)}+n_2^{(k)}-j-\tau)$-th row (note that
  $\tilde{N}_\tau^{(k,j)}$ is a square matrix).  Then, the polynomial
  \(
    \nessubres_{k,j}(F,G)
    =|\tilde{N}_j^{(k,j)}|x^j+\cdots+|\tilde{N}_0^{(k,j)}|x^0
    \)
  is called the \emph{$(k,j)$-th nested subresultant} of
  $F$ and $G$. \qed
\end{definition}

We show that the nested subresultant is equal to the recursive
subresultant up to a sign. 

\begin{theorem}
  \label{th:res-nes-subres-equiv}
  Let $F$ and $G$ be defined as in \textup{(\ref{eq:fg})}, and let
  $(P_1^{(1)},\ldots,P_{l_1}^{(1)},\ldots,P_1^{(t)},$
  $\ldots,P_{l_t}^{(t)})$
  be complete recursive PRS for $F$ and $G$ as in
  Definition~\ref{def:recprs}.  
  For $k=2,\ldots,t$ and
  $j=j_{k-1}-2,\ldots,0$, define $u_{k,j}$, $b_{k,j}$, $r_{k,j}$ and
  $R_k$ as follows:
  let 
  $u_{k,j} = (m+n-2j_1)
  \left\{
    \prod_{l=2}^{k-1}(2j_{l-1}-2j_l-1)
  \right\}
  (2j_{k-1}-2j-1)
  $ with $u_k=u_{k,j_k}$ and $u_1=m+n-2j_1$,
  $b_{k,j}=2j_{k-1}-2j-1$ with
  $b_k=b_{k,j_k}$ and $b_1=1$, 
  $r_{k,j}=(-1)^{(u_{k-1}-1)(1+2+\cdots+(b_{k,j}-1))}$ with
  $r_k=r_{k,j_k}$ and $r_{1,j}=1$ for $j<n$,
  and $R_k=(R_{k-1})^{b_k}r_k$ with $R_0=R_1=1$.
  Then, we have
  \begin{equation}
    \nessubres_{k,j}(F,G)=(R_{k-1})^{b_{k,j}}\; r_{k,j}\;
    \recsubres_{k,j}(F,G). 
  \end{equation}
\end{theorem}
To prove Theorem~\ref{th:res-nes-subres-equiv}, we prove the
following lemma.
\begin{lemma}
  For $k=1,\ldots,t$, $j=j_{k-1}-2,\ldots,0$ and $\tau=j,\ldots,0$, we
  have
  \begin{equation}
    \label{eq:res-nes-subres-equiv}
    |\tilde{N}_\tau^{(k,j)}(F,G)|=(R_{k-1})^{b_{k,j}}\; r_{k,j}\;
    |\bar{N}_\tau^{(k,j)}(F,G)|.
  \end{equation}
\end{lemma}
\begin{proof}
  By induction on $k$.  For $k=1$, it is obvious by the definitions of
  the recursive and the nested subresultants.  Assume that
  the lemma is valid for $1,\ldots,k-1$.  Then,
  for $\tau=j_{k-1},\ldots,0$, we have
  $|\tilde{N}_\tau^{(k-1,j_{k-1})}|=
  R_{k-1}|\bar{N}_\tau^{(k-1,j_{k-1})}|$. 
  For an element in recursive PRS $P_i^{(k)}(x)$, expressed as
    $P_i^{(k)}(x)=a^{(k)}_{i,n_i^{(k)}}x^{n_i^{(k)}}+\cdots
    +a^{(k)}_{i,0}x^0$,
  denote the coefficient vector for $P_i^{(k)}(x)$ by
    $\bm{p}^{(k)}_i={}^t(a^{(k)}_{i,n_i^{(k)}},\ldots,a^{(k)}_{i,0})$.
  Then, there exist certain eliminations and
  exchanges on columns which transform $\tilde{N}^{(k-1,j_{k-1})}$ to
    $\tilde{M}^{(k-1,j_{k-1})}=
    \small
    \left(
      \begin{array}{c|c}
        \tilde{W}_{k-1} & O \\
        \hline
        * & \bm{p}^{(k-1)}_{l_{k-1}}
      \end{array}
    \right)
    =
    \left(
      \begin{array}{c|c}
        \tilde{W}_{k-1} & O \\
        \hline
        * & \bm{p}^{(k)}_1
      \end{array}
    \right),$
  such that, for $\tau=j,\ldots,0$, we have 
      $\tilde{M}_\tau^{(k-1,j_{k-1})}
      =
      \small
      \left(
        \begin{array}{c|c}
          \tilde{W}_{k-1} & O \\
          \hline
          * & a^{(k-1)}_{l_{k-1},\tau}
        \end{array}
      \right)
      =
      \left(
        \begin{array}{c|c}
          \tilde{W}_{k-1} & O \\
          \hline
          * & a^{(k)}_{1,\tau}
        \end{array}
      \right)$
      with
      $|\tilde{M}_\tau^{(k-1,j_{k-1})}|
      =|\tilde{W}_{k-1}|a_{1,\tau}^{(k)},$
  where $\tilde{M}_\tau^{(k-1,j_{k-1})}$ is a sub-matrix of
  $\tilde{M}^{(k-1,j_{k-1})}$ by taking the top
  $n_1^{(k-1)}+n_2^{(k-1)}-2j_{k-1}-1$ rows and the
  $(n_1^{(k-1)}+n_2^{(k-1)}-j_{k-1}-\tau)$-th row (note that the
  matrix $\tilde{W}_{k-1}$ is a square matrix of order
  $n_1^{(k-1)}+n_2^{(k-1)}-2j_{k-1}-1$).
  By the definition of $\tilde{N}^{(k,j)}(F,G)$, we have
    $\tilde{N}^{(k,j)}(F,G)=|\tilde{W}_{k-1}|
    N^{(j)}(P_1^{(k)},P_2^{(k)})$,
  hence we have 
  \begin{equation}
    |\tilde{N}_\tau^{(k,j)}(F,G)|=|\tilde{W}_{k-1}|^{n_1^{(k)}+n_2^{(k)}-2j}
    |N_\tau^{(j)}(P_1^{(k)},P_2^{(k)})|.
  \end{equation}

  On the other hand, there exist similar transformation which
  transforms $\bar{N}^{(k-1,j_{k-1})}$ and
  $
  \small
    \left(
    \begin{array}{c}
      \bar{N}_U^{(k-1,j_{k-1})} \\
      \hline
      \bar{N}_L^{'(k-1,j_{k-1})}
    \end{array}
    \right)
  $ 
  into
  $
      \bar{M}^{(k-1,j_{k-1})} 
      =
      \small
      \left(
        \begin{array}{c|c}
          \bar{W}_{k-1} & O \\
          \hline
          * & \bm{p}^{(k-1)}_{l_{k-1}}
        \end{array}
      \right)
      =
      \left(
        \begin{array}{c|c}
          \bar{W}_{k-1} & O \\
          \hline
          * & \bm{p}^{(k)}_1
        \end{array}
      \right)
      $
   and
   $
      \bar{M}^{'(k-1,j_{k-1})} 
      =
      \small
      \left(
        \begin{array}{c|c}
          \bar{W}_{k-1} & O \\
          \hline
          * & \bm{p}^{(k)}_2
        \end{array}
      \right), $ respectively, with
      $|\bar{W}_{k-1}|=|\tilde{W}_{k-1}|$ by assumption.  Therefore,
      by exchanges of columns after the above transformations on each
      column blocks (see Terui \cite{ter2003} for detail), we have
  \begin{equation}
    \begin{split}
      (R_{k-1})^{b_{k,j}}\; r_{k,j}\;
      |\bar{N}_\tau^{(k,j)}(F,G)|&=|\bar{W}_{k-1}|^{n_1^{(k)}+n_2^{(k)}-2j}
      |N_\tau^{(j)}(P_1^{(k)},P_2^{(k)})|
      \\
      &=|\tilde{W}_{k-1}|^{n_1^{(k)}+n_2^{(k)}-2j}
      |N_\tau^{(j)}(P_1^{(k)},P_2^{(k)})|
      \\
      &
      =|\tilde{N}_\tau^{(k,j)}(F,G)|,
    \end{split}
  \end{equation}
  which proves the lemma.
  \qed
\end{proof}

\section{Reduced Nested Subresultants}
\label{sec:rednessubres}

The nested subresultant matrix has ``nested'' representation of
subresultant matrices, which makes practical use difficult.  However,
in some cases, by Gaussian elimination of the matrix with the
Sylvester's identity after some pre-computations, we can reduce the
representation of the nested subresultant matrix to ``flat''
representation, or a representation without nested determinants; this
is the reduced nested subresultant (matrix).
As we will see, the size of the reduced nested subresultant matrix
becomes much smaller than that of the recursive subresultant matrix.

First, we show the Sylvester's identity (see also Bariess
\cite{bar1968}), then explain the idea of reduction of the nested
subresultant matrix with the Sylvester's identity by an example.
\begin{lemma}[The Sylvester's Identity]
  \label{lem:sylvester}
  Let 
  $A=(a_{ij})$ be $n\times n$ matrix,
  and, for $k=1,\ldots,n-1$, $i=k+1,\ldots,n$ and $j=k+1,\ldots,n$,
  let $a_{i,j}^{(k)}=
  \small
    \begin{vmatrix}
      a_{11} & \cdots & a_{1k} & a_{1j} \\
      \vdots &        & \vdots & \vdots \\
      a_{k1} & \cdots & a_{kk} & a_{kj} \\
      a_{i1} & \cdots & a_{ik} & a_{ij} \\
    \end{vmatrix}
    $.
  Then, we have
  $
  |A|
  \left(
    a_{kk}^{(k-1)}
  \right)^{n-k-1}
  =
  \small
    \begin{vmatrix}
      a_{k+1,k+1}^{(k)} & \cdots & a_{k+1,n}^{(k)} \\
      \vdots & & \vdots \\
      a_{n,k+1}^{(k)} & \cdots & a_{n,n}^{(k)}
    \end{vmatrix}
    .
    $
  \qed
\end{lemma}

\begin{example}
  Let $F(x)$ and $G(x)$ be defined as
  \begin{equation}
      F(x)
      = a_6 x^6+a_5 x^5+\cdots+a_0,
      \quad
      G(x) 
      = b_5 x^5+b_4 x^4+\cdots+b_0,
  \end{equation}
  with $a_6\ne 0$ and $b_5\ne 0$.  We assume that vectors of
  coefficients $(a_6,a_5)$ and $(b_5,b_4)$ are linearly independent as
  vectors over $K$, and that $\prs(F,G)=(P_1^{(1)}=F,\ P_2^{(1)}=G,\ 
  P_3^{(1)}=\gcd(F,G))$ with $\deg(P_3^{(1)})=4$.  Consider the
  $(2,2)$-th nested subresultant; its matrix is defined as
  \begin{equation}
    \label{eq:aj}
    \small
    \tilde{N}^{(2,2)}=
    \begin{pmatrix}
      A_4 & 4A_4 & \\
      A_3 & 3A_3 & 4A_4 \\
      A_2 & 2A_2 & 3A_3 \\
      A_1 & A_1  & 2A_2 \\
      A_0 &      & A_1
    \end{pmatrix}
    ,
    \quad
    A_j=
    \begin{vmatrix}
      a_6 & b_5 & \\
      a_5 & b_4 & b_5 \\
      a_j & b_{j-1} & b_j
    \end{vmatrix}
    ,
  \end{equation}
  for $j\le 4$ with $b_j=0$ for $j<0$.  Now, let us calculate the
  leading coefficient of $\tilde{\subres}_{2,2}(F,G)$ as
  \begin{equation}
    \label{eq:ex-h1}
    \small
    \begin{split}
    |\tilde{N}^{(2,2)}_{2}|=
    \begin{vmatrix}
      A_4 & 4A_4 & \\
      A_3 & 3A_3 & 4A_4 \\
      A_2 & 2A_2 & 3A_3
    \end{vmatrix}
    &=
      \begin{vmatrix}
        \begin{vmatrix}
          a_6 & b_5 \\
          a_5 & b_4 & b_5 \\
          a_4 & b_3 & b_4
        \end{vmatrix}
        &
        \begin{vmatrix}
          a_6 & b_5 & \\
          a_5 & b_4 & b_5 \\
          4a_4 & 4b_3 & 4b_4
        \end{vmatrix}
        &
        \begin{vmatrix}
          a_6 & b_5 & \\
          a_5 & b_4 & b_5 \\
          0a_4 & 0b_3 & 0b_4
        \end{vmatrix}
        \\
        \noalign{\vskip2pt}
        \begin{vmatrix}
          a_6 & b_5 \\
          a_5 & b_4 & b_5 \\
          a_3 & b_2 & b_3
        \end{vmatrix}
        &
        \begin{vmatrix}
          a_6 & b_5 \\
          a_5 & b_4 & b_5 \\
          3a_3 & 3b_2 & 3b_3
        \end{vmatrix}
        &
        \begin{vmatrix}
          a_6 & b_5 & \\
          a_5 & b_4 & b_5 \\
          4a_4 & 4b_3 & 4b_4
        \end{vmatrix}
        \\
        \noalign{\vskip2pt}
        \noalign{\vskip2pt}
        \begin{vmatrix}
          a_6 & b_5 \\
          a_5 & b_4 & b_5 \\
          a_2 & b_1 & b_2
        \end{vmatrix}
        &
        \begin{vmatrix}
          a_6 & b_5 \\
          a_5 & b_4 & b_5 \\
          2a_2 & 2b_1 & 2b_2
        \end{vmatrix}
        &
        \begin{vmatrix}
          a_6 & b_5 & \\
          a_5 & b_4 & b_5 \\
          3a_3 & 3b_2 & 3b_3
        \end{vmatrix}
      \end{vmatrix}
      \\
    &=|H|
    =
    \left|
      \begin{pmatrix}
        H_{p,q}
      \end{pmatrix}
    \right|
    .
    \end{split}
  \end{equation}
  To apply the Sylvester's identity on H, we make the $(3,1)$ and the
  $(3,2)$ elements in $H_{p,2}$ and $H_{p,3}$ ($p=1,2,3$) equal to
  those elements in $H_{p,1}$, respectively, by adding the first and
  the second rows, multiplied by certain numbers, to the third row.
  For example, in $H_{1,2}$, calculate $x_{12}$ and $y_{12}$ by
  solving a system of linear equations
  \begin{equation}
    \label{eq:h12}
    \left\{
      \begin{split}
        a_6 x_{12} + a_5 y_{12} &= -4a_4+a_4=-3a_4 \\
        b_5 x_{12} + b_4 y_{12} &= -4b_3+b_3=-3b_3
      \end{split}
    \right.
    ,
  \end{equation}
  (Note that (\ref{eq:h12}) has a solution in $K$ by assumption), then
  add the first row multiplied by $x_{12}$ and the second row
  multiplied by $y_{12}$, respectively, to the third row.  Then, we
  have $
    H_{1,2}=
    \begin{vmatrix}
      a_6 & b_5 \\
      a_5 & b_4 & b_5 \\
      a_4 & b_3 & h_{12}
     \end{vmatrix}
     $ with $
     h_{12} = 4 b_4 + y_{12}b_5.
     $ Doing similar calculations for the other $H_{p,q}$, we
     calculate $h_{p,q}$ for $H_{p,q}$ similarly as in the above.
     Finally, by the Sylvester's identity, we have
  \begin{equation}
    |\tilde{N}^{(2,2)}_2|=
    \begin{vmatrix}
      a_6 & b_5 \\
      a_5 & b_4
    \end{vmatrix}
    ^2
    \begin{vmatrix}
      a_6 & b_5 & \\
      a_5 & b_4 & b_5 & b_5 & b_5\\
      a_4 & b_3 & b_4 & h_{12} & h_{13}\\
      a_3 & b_2 & b_3 & h_{22} & h_{23}\\
      a_2 & b_1 & b_2 & h_{32} & h_{33}
    \end{vmatrix}
    =
    \begin{vmatrix}
      a_6 & b_5 \\
      a_5 & b_4
    \end{vmatrix}
    ^2
    |\hat{N}^{(2,2)}_2|
    ,
  \end{equation}
  note that we have derived $\hat{N}^{(2,2)}_2$ as a reduced form of
  $\tilde{N}^{(2,2)}_2$.  \qed
\end{example}

\begin{definition}[Reduced Nested Subresultant Matrix]
  \label{def:rednessubresmat}
  Let $F$ and $G$ be defined as in \textup{(\ref{eq:fg})}, and let
  $(P_1^{(1)},\ldots,P_{l_1}^{(1)},\ldots,$
  $P_1^{(t)},\ldots,P_{l_t}^{(t)})$ be complete recursive PRS for $F$
  and $G$ as in Definition~\ref{def:recprs}.  Then, for each tuple of
  numbers $(k,j)$ with $k=1,\ldots,t$ and $j=j_{k-1}-2,\ldots,0$,
  define matrix $\hat{N}^{(k,j)}(F,G)$ recursively as follows.
  \begin{enumerate}
  \item For $k=1$, let $\hat{N}^{(1,j)}(F,G)=N^{(j)}(F,G)$.
  \item For $k>1$, let $\hat{N}_U^{(k-1,j_{k-1})}(F,G)$ be a
    sub-matrix of $\hat{N}^{(k-1,j_{k-1})}(F,G)$ by deleting the
    bottom $j_{k-1}+1$ rows, and $\hat{N}_L^{(k-1,j_{k-1})}(F,G)$ be a
    sub-matrix of $\hat{N}^{(k-1,j_{k-1})}(F,G)$ by taking the bottom
    $j_{k-1}+1$ rows, respectively.  For $\tau=j_{k-1},\ldots,0$ let
    $\hat{N}_\tau^{(k-1,j_{k-1})}(F,G)$ be a sub-matrix of
    $\hat{N}^{(k-1,j_{k-1})}(F,G)$ by putting
    $\hat{N}_U^{(k-1,j_{k-1})}(F,G)$ on the top and the
    $(j_{k-1}-\tau+1)$-th row of\linebreak
    $\hat{N}_L^{(k-1,j_{k-1})}(F,G)$ in
    the bottom row.  Let
    $\hat{A}_\tau^{(k-1)}=|\hat{N}_\tau^{(k-1,j_{k-1})}|$ and
    construct a matrix $H$ as
  \begin{equation}
    \label{eq:h1}
    H =
    \begin{pmatrix}
      H_{p,q}
    \end{pmatrix}
    =
    N^{(j)}
    \left(
      \hat{A}^{(k-1)}(x), \frac{d}{dx}\hat{A}^{(k-1)}(x)
    \right),
  \end{equation}
  where
    $\hat{A}^{(k-1)}(x)=\hat{A}^{(k-1)}_{j_{k-1}}x^{j_{k-1}}+\cdots
    +\hat{A}^{(k-1)}_0x^0.$
  Since $\hat{N}_\tau^{(k-1,j_{k-1})}$ consists of
  $\hat{N}_U^{(k-1,j_{k-1})}$ and a row vector in the bottom, we
  express $\hat{N}_U^{(k-1,j_{k-1})}=
  \begin{pmatrix}
    U^{(k)} | \bm{v}^{(k)}
  \end{pmatrix}
  $, where $U^{(k)}$ is a square matrix and $\bm{v}^{(k)}$ is a column
  vector, and the row vector by $
  \begin{pmatrix}
    \bm{b}_{p,q}^{(k)} \Bigm| g_{p,q}^{(k)}
  \end{pmatrix}
  $, where $\bm{b}_{p,q}^{(k)}$ is a row vector and $g_{p,q}^{(k)}$ is
  a number, respectively, such that
  \begin{equation}
    \label{eq:h1pq}
    H_{p,q}=
    \left|
      \begin{array}{c|c}
        U^{(k)} & \bm{v}^{(k)} \\
        \hline
        \bm{b}_{p,q}^{(k)} & g_{p,q}^{(k)}
      \end{array}
    \right|
    ,
  \end{equation}
  with $\bm{b}_{p,q}^{(k)}=\bm{0}$ and $g_{p,q}^{(k)}=0$ for
  $H_{p,q}=0$.  Furthermore, we assume that $U^{(k)}$ is not singular.
  Then, for $p=1,\ldots,n_1^{(k)}+n_2^{(k)}-j$ and
  $q=2,\ldots,n_1^{(k)}+n_2^{(k)}-j$, calculate a row vector
  $\bm{x}_{p,q}^{(k)}$ as a solution of the equation
    $\bm{x}_{p,q}^{(k)}U^{(k)}=\bm{b}_{p,1}^{(k)}$,
  and define $h_{p,q}^{(k)}$ as
    $h_{p,q}^{(k)}=\bm{x}_{p,q}^{(k)}\bm{v}^{(k)}$.
  Finally, define $\hat{N}^{(k,j)}(F,G)$ as
  \begin{gather}
    \label{eq:rednessubresmat}
    \hat{N}^{(k,j)}(F,G)=
    \begin{pmatrix}
      U^{(k)} & \bm{v}^{(k)} & \bm{v}^{(k)} & \cdots & \bm{v}^{(k)} \\
      \bm{b}_{1,1}^{(k)} & g_{1,1}^{(k)} & h_{1,2}^{(k)} & \cdots &
      h_{1,J_{k,j}}^{(k)} \\
      \bm{b}_{2,1}^{(k)} & g_{2,1}^{(k)} & h_{2,2}^{(k)} & \cdots &
      h_{2,J_{k,j}}^{(k)} \\
      \vdots & \vdots & \vdots & & \vdots \\
      \bm{b}_{I_{k,j},1}^{(k)} & g_{I_{k,j},1}^{(k)} &
      h_{I_{k,j},2}^{(k)} & \cdots & h_{I_{k,j},J_{k,j}}^{(k)}
    \end{pmatrix}
    ,
    \\
    \label{eq:ikj-jkj}
    \begin{split}
      I_{k,j} 
      &
      = n_1^{(k)}+n_2^{(k)}-j=(2j_{k-1}-2j-1)+j,
      \\
      J_{k,j} 
      &
      = n_1^{(k)}+n_2^{(k)}-2j=2j_{k-1}-2j-1.
    \end{split}
    \end{gather}
  Then, $\hat{N}^{(k,j)}(F,G)$ is called the \emph{$(k,j)$-th reduced
    nested subresultant matrix of $F$ and $G$}. \qed
  \end{enumerate}
\end{definition}
\begin{proposition}
  \label{pr:rednessubresmatord}
  For $k=1,\ldots,t$ and $j<j_{k-1}-1$, the numbers of rows and
  columns of the $(k,j)$-th reduced nested subresultant matrix
  $\hat{N}^{(k,j)}(F,G)$ are
      $(m+n-2(k-1)-2j)+j$ and
      $(m+n-2(k-1)-2j)$,
  respectively.
\end{proposition}
\begin{proof}
  By induction on $k$.  It is obvious for $k=1$. Assume
  that the proposition is valid for $1,\ldots,k-1$.
  Then, the numbers of rows and columns of matrix $U^{(k)}$ in
  (\ref{eq:rednessubresmat}) are equal to
    $(m+n-2\{(k-1)-1\}-2j_{k-1})-1$,
  respectively.  Therefore, by (\ref{eq:rednessubresmat}) and
  (\ref{eq:ikj-jkj}), we prove the proposition for $k$.
  \qed
\end{proof}
Note that, as Proposition~\ref{pr:rednessubresmatord} shows, the size
of the reduced nested subresultant matrix, which is at most the sum of
the degree of the initial polynomials, is much smaller than that of
the recursive subresultant matrix (see
Proposition~\ref{prop:recsubresmat}).
\begin{definition}[Reduced Nested Subresultant]
  \label{def:rednessubres}
  Let $F$ and $G$ be defined as in \textup{(\ref{eq:fg})}, and let
  $(P_1^{(1)},\ldots,P_{l_1}^{(1)},\ldots,P_1^{(t)},\ldots,P_{l_t}^{(t)})$
  be complete recursive PRS for $F$ and $G$ as in
  Definition~\ref{def:recprs}.  For $j=j_{k-1}-2,\ldots,0$ and
  $\tau=j,\ldots,0$, let
  $\hat{N}_\tau^{(k,j)}=\hat{N}_\tau^{(k,j)}(F,G)$ be a sub-matrix of
  the $(k,j)$-th reduced nested subresultant matrix
  $\hat{N}^{(k,j)}(F,G)$ obtained by the top $m+n-2(k-1)-2j-1$ rows
  and the $(m+n-2(k-1)-j-\tau)$-th row (note that
  $\hat{N}_\tau^{(k,j)}(F,G)$ is a square matrix).  Then, the
  polynomial
  \(
    \rednessubres_{k,j}(F,G)=|\hat{N}_j^{(k,j)}(F,G)|x^j+\cdots
    +|\hat{N}_0^{(k,j)}(F,G)|x^0
    \)
  is called the $(k,j)$-th reduced nested subresultant of $F$ and $G$.
  \qed
\end{definition}

Now, we derive the relationship between the nested and the reduced
nested subresultants.

\begin{theorem}
  \label{th:nessubres-rednessubres-equiv}
  Let $F$ and $G$ be defined as in \textup{(\ref{eq:fg})}, and let
  $(P_1^{(1)},\ldots,P_{l_1}^{(1)},\ldots,P_1^{(t)},$
  $\ldots,P_{l_t}^{(t)})$
  be complete recursive PRS for $F$ and $G$ as in
  Definition~\ref{def:recprs}.  For $k=2,\ldots,t$,
  $j=j_{k-1}-2,\ldots,0$ with $J_{k,j}$ as in
  (\ref{eq:rednessubresmat}), define $\hat{B}_{k,j}$ and $\hat{R}_k$
  as
    $\hat{B}_{k,j}=|U^{(k)}|^{J_{k,j}-1}$
  with $\hat{B}_k=\hat{B}_{k,j_k}$ and $\hat{B}_1=\hat{B}_2=1$, and
    $\hat{R}_k = (\hat{R}_{k-1}\cdot\hat{B}_{k-1})^{J_{k,j_k}}$
  with $\hat{R}_1=\hat{R}_2=1$, respectively.  Then, we have
  \begin{equation}
    \nessubres_{k,j}(F,G)=
    (\hat{R}_{k-1}\cdot\hat{B}_{k-1})^{J_{k,j}}\hat{B}_{k,j}\cdot
    \rednessubres_{k,j}(F,G).
  \end{equation}
\end{theorem}
To prove Theorem~\ref{th:nessubres-rednessubres-equiv}, we prove the
following lemma.
\begin{lemma}
  \label{le:nessubres-rednessubres-equiv}
  For $k=1,\ldots,t$, $j=j_{k-1}-2,\ldots,0$ and $\tau=j,\ldots,0$, we
  have
  \begin{equation}
    \label{eq:nessubresmat-rednessubresmat-equiv}
    |\tilde{N}_\tau^{(k,j)}(F,G)|=
    (\hat{R}_{k-1}\cdot\hat{B}_{k-1})^{J_{k,j}}\hat{B}_{k,j}
    |\hat{N}_\tau^{(k,j)}(F,G)|.
  \end{equation}
\end{lemma}
\begin{proof}
  By induction on $k$.  For $k=1$, it is obvious from the definitions
  of the nested and the reduced nested subresultants.  Assume that
  the lemma is valid for $1,\ldots,k-1$.  Then, for
  $\tau=j_{k-1},\ldots,0$, we have
  \begin{equation}
    \begin{split}
      |\tilde{N}_\tau^{(k-1,j_{k-1})}(F,G)|
    =&
    (\hat{R}_{k-2}\cdot\hat{B}_{k-2})^{J_{k-1,j_{k-1}}}\hat{B}_{{k-1},j_{k-1}}
    |\hat{N}_\tau^{(k-1,j_{k-1})}(F,G)|
    \\
    =&
    (\hat{R}_{k-1}\cdot\hat{B}_{k-1})|\hat{N}_\tau^{(k-1,j_{k-1})}(F,G)|.
    \end{split}
  \end{equation}
  Let 
    $\tilde{A}_\tau^{(k-1)}=|\tilde{N}_\tau^{(k-1,j_{k-1})}|$ and
    $\hat{A}_\tau^{(k-1)}=|\hat{N}_\tau^{(k-1,j_{k-1})}|$.
  Then, by the definition of the $(k,j)$-th nested subresultant, we
  have
  \begin{align}
    |\tilde{N}_\tau^{(k,j)}|
    =&
    \label{eq:nesresmat-tau}
    \tiny{
    \begin{vmatrix}
      \tilde{A}_{j_{k-1}}^{(k-1)} &        &       &
      j_{k-1}\tilde{A}_{j_{k-1}}^{(k-1)}    &        & \\
      \vdots & \ddots &  & \vdots & \ddots & \\
      \vdots &        & \tilde{A}_{j_{k-1}}^{(k-1)} &
      \vdots &        & j_{k-1}\tilde{A}_{j_{k-1}}^{(k-1)}\\
      \vdots &        & \vdots & \vdots &        & \vdots \\
      \tilde{A}_{2j-j_{k-1}+3}^{(k-1)} & \cdots &
      \tilde{A}_{j+1}^{(k-1)} &
      (2j-j_{k-1}+3)\tilde{A}_{2j-j_{k-1}+3}^{(k-1)} & \cdots &
      (j+2)\tilde{A}_{j+2}^{(k-1)} \\
      \hline
      \tilde{A}_{j-j_{k-1}+\tau+2}^{(k-1)} & \cdots &
      \tilde{A}_\tau^{(k-1)} &
      (j-j_{k-1}+\tau+2)\tilde{A}_{j-j_{k-1}+\tau+2}^{(k-1)} &
      \cdots & (\tau+1)\tilde{A}_{\tau+1}^{(k-1)}  
    \end{vmatrix}
  }
    \\
    \label{eq:nesresmat-tau-2}
    =&
    (\hat{R}_{k-1}\cdot\hat{B}_{k-1})^{J_{k,j}}|H'|,
  \end{align}
  where $\tilde{A}_{l}^{(k-1)}=0$ for $l<0$ and 
  $H'=
  \begin{pmatrix}
    H'_{p,q}
  \end{pmatrix}
  $ is defined as (\ref{eq:nesresmat-tau}) with
  $\tilde{A}_{l}^{(k-1)}$ replaced by $\hat{A}_{l}^{(k-1)}$ (note
  that $\tilde{N}_\tau^{(k,j)}$ and $H'$ are square matrices of order
  $J_{k,j}$).  Then, by Definition~\ref{def:rednessubresmat}, we can
  express $H'_{p,q}$ as
  $\displaystyle{
    H'_{p,q}=
    \left|
      \begin{array}{c|c}
        U^{(k)} & \bm{v}^{(k)} \\
        \hline
        \bm{b}'^{(k)}_{p,q} & g'^{(k)}_{p,q}
      \end{array}
    \right|
  }$
  with $\bm{b}'^{(k)}_{p,q}=\bm{0}$ and $g'^{(k)}_{p,q}=0$ for
  $H'_{p,q}=0$.  Note that, for $q=1,\ldots,J_{k,j}$, we have
  $\bm{b}'^{(k)}_{p,q}=\bm{b}^{(k)}_{p,q}$ and
  $g'^{(k)}_{p,q}=g^{(k)}_{p,q}$ for $p=1,\ldots,J_{k,j}-1$, and
  $\bm{b}'^{(k)}_{J_{k,j},q}=\bm{b}^{(k)}_{I_{k,j}-\tau,q}$ and
  $g'^{(k)}_{J_{k,j},q}=g^{(k)}_{I_{k,j}-\tau,q}$, where
  $\bm{b}^{(k)}_{p,q}$ and $g^{(k)}_{p,q}$ are defined in
  (\ref{eq:h1pq}), respectively.  Furthermore, by the definition of
  $h_{p,q}^{(k)}$ in Definition~\ref{def:rednessubresmat}, we have
  \begin{equation}
    |H'|=
    \small
    \begin{vmatrix}
      \left|
        \begin{array}{c|c}
          U^{(k)} & \bm{v}^{(k)} \\
          \hline
          \bm{b}'^{(k)}_{1,1} & g^{(k)}_{1,1}
        \end{array}
      \right|
      &
      \left|
        \begin{array}{c|c}
          U^{(k)} & \bm{v}^{(k)} \\
          \hline
          \bm{b}'^{(k)}_{1,1} & h^{(k)}_{1,2}
        \end{array}
      \right|
      &
      \cdots
      &
      \left|
        \begin{array}{c|c}
          U^{(k)} & \bm{v}^{(k)} \\
          \hline
          \bm{b}'^{(k)}_{1,1} & h^{(k)}_{1,J_{k,j}}
        \end{array}
      \right|
      \\
      \vdots & \vdots & & \vdots \\
      \left|
        \begin{array}{c|c}
          U^{(k)} & \bm{v}^{(k)} \\
          \hline
          \bm{b}'^{(k)}_{J_{k,j}-1,1} & g^{(k)}_{J_{k,j}-1,1}
        \end{array}
      \right|
      &
      \left|
        \begin{array}{c|c}
          U^{(k)} & \bm{v}^{(k)} \\
          \hline
          \bm{b}'^{(k)}_{J_{k,j}-1,1} & h^{(k)}_{J_{k,j}-1,2}
        \end{array}
      \right|
      &
      \cdots
      &
      \left|
        \begin{array}{c|c}
          U^{(k)} & \bm{v}^{(k)} \\
          \hline
          \bm{b}'^{(k)}_{J_{k,j}-1,1} & h^{(k)}_{J_{k,j}-1,J_{k,j}}
        \end{array}
      \right|
      \\
      \noalign{\vskip2pt}
      \left|
        \begin{array}{c|c}
          U^{(k)} & \bm{v}^{(k)} \\
          \hline
          \bm{b}'^{(k)}_{I_{k,j}-\tau,1} & g^{(k)}_{I_{k,j}-\tau,1}
        \end{array}
      \right|
      &
      \left|
        \begin{array}{c|c}
          U^{(k)} & \bm{v}^{(k)} \\
          \hline
          \bm{b}'^{(k)}_{I_{k,j}-\tau,1} & h^{(k)}_{I_{k,j}-\tau,2}
        \end{array}
      \right|
      &
      \cdots
      &
      \left|
        \begin{array}{c|c}
          U^{(k)} & \bm{v}^{(k)} \\
          \hline
          \bm{b}'^{(k)}_{I_{k,j}-\tau,1} & h^{(k)}_{I_{k,j}-\tau,J_{k,j}}
        \end{array}
      \right|
    \end{vmatrix}
    .
  \end{equation}
  By Lemma~\ref{lem:sylvester}, we have 
  \begin{equation}
    \label{eq:nesresmat-tau-2-reduced}
    |H'|=|U^{(k)}|^{J_{k,j}-1}|\hat{N}^{(k,j)}_\tau(F,G)|
    =\hat{B}_{k,j}|\hat{N}^{(k,j)}_\tau(F,G)|,
  \end{equation}
  hence, by putting (\ref{eq:nesresmat-tau-2-reduced}) into
  (\ref{eq:nesresmat-tau-2}), we prove the lemma.
  \qed
\end{proof}
\begin{remark}
  \label{re:time}
  We can estimate arithmetic computing time for the $(k,j)$-th reduced
  nested resultant matrix $\hat{N}^{(k,j)}$ in
  \textup{(\ref{eq:rednessubresmat})}, as follows.  The computing time
  for the elements $h_{p,q}$ is dominated by the time for Gaussian
  elimination of $U^{(k)}$.  Since the order of $U^{(k)}$ is equal to
  $m+n-2(k-2)-2j_{k-1}$ (see Proposition~\ref{pr:rednessubresmatord}),
  it is bounded by $O((m+n-2(k-2)-2j_{k-1})^3)$, or $O((m+n)^3)$ (see
  Golub and van Loan \textup{\cite{gol-vloa1996}} for example).  
  We can calculate
  $\hat{N}^{(k,j)}(F,G)$ for $j<j_{k-1}-2$ by $\hat{N}^{(k,0)}(F,G)$,
  hence the total computing time for $\hat{N}^{(k,j)}$ for the entire
  recursive PRS ($k=1,\ldots,t$) is bounded by $O(t(m+n)^3)$ (see also
  for the conclusion). \qed
\end{remark}

\section{Conclusion and Motivation}
\label{sec:disc}

In this paper, we have given two new expressions of subresultants for
the recursive PRS, the nested subresultant and the reduced nested
subresultant.  We have shown that the reduced nested subresultant
matrix reduces the size of the matrix drastically to at most the sum
of the degree of the initial polynomials compared with the recursive
subresultant matrix.  We have also shown that we can calculate the
reduced nested subresultant matrix by solving certain systems of
linear equations of order at most the sum of the degree of the initial
polynomials.

A main limitation of the reduced nested subresultant in this paper is
that we cannot calculate its matrix in the case the matrix $U^{(k)}$
in (\ref{eq:h1pq}) is singular.  We need to develop a method to
calculate the reduced nested subresultant matrix in the case such that
$U^{(k)}$ is singular in general.

From a point of view of computational complexity, the algorithm for
the reduced nested subresultant matrix has a cubic complexity bound in
terms of the degree of the input polynomials (see
Remark~\ref{re:time}).  However, subresultant algorithms which have a
quadratic complexity bound in terms of the degree of the input
polynomials have been proposed (\cite{duc2000},
\cite{lom-roy-eldin2000}); the algorithms exploit the structure of the
Sylvester matrix to increase their efficiency with controlling the
size of coefficients well.  Although, in this paper, we have primarily
focused our attention into reducing the structure of the nested
subresultant matrix to ``flat'' representation, development of more
efficient algorithm such as exploiting the structure of the Sylvester
matrix would be the next problem.  Furthermore, the reduced nested
subresultant may involve fractions which may be unusual for
subresultants, hence more detailed analysis of computational
efficiency including comparison with (ordinary and recursive)
subresultants would also be necessary.

We expect that the reduced nested subresultants can be used for
including approximate algebraic computation, especially for the
square-free decomposition of approximate univariate polynomials with
approximate GCD computations based on Singular Value Decomposition
(SVD) of subresultant matrices (\cite{cor-gia-tra-wat1995}
\cite{emi-gal-lom1997}), which motivates the present work.  We can
calculate approximate square-free decomposition of the given
polynomial $P(x)$ by several methods including calculation of the
approximate GCDs of $P(x),\ldots,P^{(n)}(x)$ (by $P^{(n)}(x)$ we
denote the $n$-th derivative of $P(x)$) or those of the recursive PRS
for $P(x)$ and $P'(x)$; as for these methods, we have to find the
representation of the subresultant matrices for
$P(x),\ldots,P^{(n)}(x)$, or that for the recursive PRS for $P(x)$ and
$P'(x)$, respectively.  While several algorithms based on different
representation of subresultant matrices have been proposed
(\cite{dtoc-gveg2002} \cite{rup1999}) for the former approach, we
expect that our reduced nested subresultant matrix can be used for the
latter approach.  To make use of the reduced nested subresultant
matrix, we need to reveal the relationship between the structure of
the subresultant matrices and their singular values; this is the
problem on which we are working now.

\bibliographystyle{splncs}

\end{document}